\documentclass[12pt]{article}
\usepackage{amsmath,amsfonts,amssymb,amsthm,color}
\usepackage{graphicx}
\usepackage[numbers]{natbib}
\usepackage{url}

\usepackage{tikz}

\newcommand{\drbl}[2]{\draw [thin](V#1)--(V#2);}
\newcommand{\drwl}[2]{\draw[thin,dotted] (V#1)--(V#2);}

\newcommand{\Z}{{\mathbb Z}}

\newcommand{\cS}{{\cal S}}

\newtheorem{theorem}{Theorem}
\newtheorem{lemma}{Lemma}

\newtheorem{proposition}{Proposition}

\usepackage{bm}
\bmdefine{\Bt}{t}
\bmdefine{\BX}{X}
\bmdefine{\BY}{Y}
\bmdefine{\Bb}{b}
\bmdefine{\Bx}{x}
\bmdefine{\By}{y}
\bmdefine{\Bz}{z}

\newcommand{\cF}{{\cal F}}

\begin{document}

\title{A Markov basis for two-state toric homogeneous Markov chain model
without initial parameters}


\author{
Hisayuki Hara\footnote{
Department of Technology Management for Innovation, 
University of Tokyo} \ and 
Akimichi Takemura\footnote{
Graduate School of Information Science and Technology, 
University of Tokyo}
\footnote{JST, CREST}
}

\date{September, 2010}

\maketitle

\begin{abstract}
We derive  a  Markov basis consisting of moves of degree at most three 
for two-state toric homogeneous Markov chain model of arbitrary length 
without parameters for initial states.   
Our basis consists of moves of degree three and degree one, which
alter the initial frequencies, in addition  to moves of degree two and
 degree one  
for toric homogeneous Markov chain model with parameters for initial states.
\end{abstract}


\section{Introduction}

Consider a Markov chain $X_t$, $t=1,\dots,T\,(\ge 3)$, over finite state
space $\cS$.  
Let $\omega=(s_1, \dots, s_T)\in \cS^T$ denote a path of a Markov
chain.  
In this paper we discuss Markov bases of toric ideals arising from the
following statistical model 
\begin{equation}
\label{eq:model}
p(\omega)=c \beta_{s_1 s_2} \dots \beta_{s_{T-1} s_T},
\end{equation}
where $c$ is a normalizing constant.
In \citet{ascb}, the model (\ref{eq:model}) is called a toric Markov chain
model.  
In \citet{th-homogeneous1} we considered a model with additional parameters
$\gamma_s, s\in \cS,$ for the initial states:
\begin{equation}
\label{eq:model0}
p(\omega)=c \gamma_{s_1} \beta_{s_1 s_2} \dots \beta_{s_{T-1} s_T}
\end{equation}
and derived a Markov basis of toric ideals arising from
(\ref{eq:model0}) for the case of $\cS=\{1,2\}$ (arbitrary $T$) and the
case of $T=3$ (arbitrary $\cS$).   
In \citet{th-homogeneous1} we called the model (\ref{eq:model0}) 
a toric homogeneous Markov chain(THMC) model. 
The model (\ref{eq:model}) corresponds to THMC model with 
$\gamma_1 = \cdots = \gamma_{|\cS|}$. 
For distinguishing two models, we call \eqref{eq:model} 
a THMC model without initial parameters.

In the present paper, we generalize the result in \citet{th-homogeneous1} 
to the model (\ref{eq:model})  and derive a Markov
basis for the case $\cS=\{1,2\}$ and arbitrary $T$. 

From a statistical viewpoint, Markov bases are used to test
goodness-of-fit of the model based on the exact distribution of a test
statistic.  
A data set of paths is summarized in an $|\cS|^T$ contingency table. 
The set of contingency tables sharing sufficient statistic is called a
fiber. 
A Markov basis is defined as a set of moves connecting every fiber.  
In this paper we consider the problem in the framework of contingency
table analysis and derive a Markov basis as a set of moves which
guarantees the connectivity of every fiber. 


The organization of the paper is as follows. 
In Section 2, we introduce some notations and terminologies and give
some preliminary results. 
In Section 3, we state the main theorem and give a Markov basis for 
the model (\ref{eq:model}) with $\cS=\{1,2\}$. 
We give a proof of the theorem in Section 4.
A numerical example with a real data set is given in Section 5.
In Section 6 we end the paper with some concluding remarks. 

\section{Preliminaries}
\subsection{Notation and terminology}

Let $\omega=(s_1, \dots, s_T)$ be a path.  For notational simplicity
we sometimes write $\omega=(s_1\dots s_T)$ or $\omega=s_1\dots s_T$.
If $s_1=\dots=s_T=i$, we call $\omega$ a flat path at state $i$.
If for some $t$ $(1\le t < T)$, $s_1=\dots=s_t=i \neq j = s_{t+1}=\dots=s_T$, 
then we call $\omega$ a single-step path  from $i$ to $j$ at time $t$.
We say that a path $\omega$ starts at $i$ and ends at $j$ if
$s_1=i, s_T=j$.  The set of such paths  is denoted by $W_{i*\dots*j}$.
We call $\omega$ a non-flat cycle, if
it starts from $i$, visits $j\neq i$ on the way and ends at $i$.  We denote
the set of non-flat cycles starting from $i$ by
\begin{equation}
\label{eq:iji}
W_{i*j*i}=\{ \omega \mid s_1=i, s_t=j, s_T=i, \ \text{for some }   1 < t <  T\}.
\end{equation}
In Figure \ref{fig:flat+single-step}, we depict examples of a flat path,
a single-step path and a non-flat cycle. 

\begin{figure}[htbp]
 \centering
 \begin{tabular}{ccc}
  \begin{tikzpicture}[baseline=-1.2cm,scale=0.8]
   \foreach \i in {1,2} {   \foreach \j in {1,...,4}    {
   \path (\j,-\i) coordinate (V\i\j);  \fill (V\i\j) circle (1pt); }}
   \draw (V11)--(V12)--(V13)--(V14);
   \draw (0,-1) node {1};
   \draw (0,-2) node {2};
   \draw (0,-0.4) node {$s\backslash t$};
   \foreach \j in {1,...,4} { \draw (V1\j)  node [above]  {\j};}
  \end{tikzpicture} & 
  \begin{tikzpicture}[baseline=-1.2cm,scale=0.8]
   \foreach \i in {1,2} {   \foreach \j in {1,...,4}    {
   \path (\j,-\i) coordinate (V\i\j);  \fill (V\i\j) circle (1pt); }}
   \draw (V11)--(V12)--(V23)--(V24);
   \draw (0,-1) node {1};
   \draw (0,-2) node {2};
   \draw (0,-0.4) node {$s\backslash t$};
   \foreach \j in {1,...,4} { \draw (V1\j)  node [above]  {\j};}
  \end{tikzpicture}  \vspace{0.2cm} & 
  \begin{tikzpicture}[baseline=-1.2cm,scale=0.8]
   \foreach \i in {1,2} {   \foreach \j in {1,...,4}    {
   \path (\j,-\i) coordinate (V\i\j);  \fill (V\i\j) circle (1pt); }}
   \draw (V11)--(V22)--(V23)--(V14);
   \draw (0,-1) node {1};
   \draw (0,-2) node {2};
   \draw (0,-0.4) node {$s\backslash t$};
   \foreach \j in {1,...,4} { \draw (V1\j)  node [above]  {\j};}
  \end{tikzpicture}  \vspace{0.2cm}\\
  (i) a flat path & (ii) a single-step path & (iii) a non-flat cycle
\end{tabular}
\caption{A flat path,  a single-step path and a non-flat cycle for $T=4$}
 \label{fig:flat+single-step}
\end{figure}
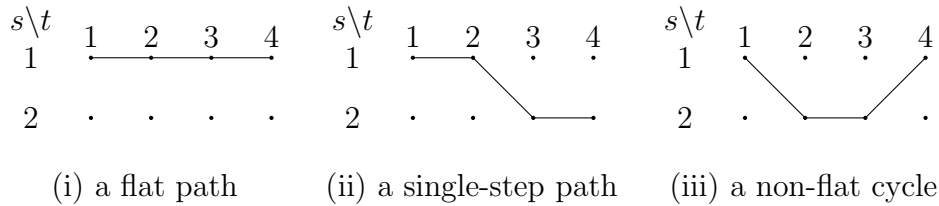

A data set of $n$ paths is summarized in an $|\cS|^T$ contingency
table $\Bx=\{x(\omega), \omega\in \cS^T\}$ of total frequency $n$, where
$x(\omega)$ denotes the frequency of the path $\omega$.
Let $x^{t}_{ij}=\sum_{\omega: s_t=i, s_{t+1}=j} x(\omega)$ 
denote the number of transitions from $s_t=i$ to
$s_{t+1}=j$ in $\Bx$ and let $x^{t}_i=\sum_{\omega: s_t=i} x(\omega)$ denote the frequency of the
state $s_t=i$ in $\Bx$. 
In particular $x^1_i$ is the frequency of the 
initial state $s_1=i$.
Let
\[
x^+_{ij}=\sum_{t=1}^{T-1} x^t_{ij}
\]
denote the total number of transitions from $i$ to $j$ in $\Bx$.

The sufficient statistic for the model \eqref{eq:model}
is given by the frequencies of transitions
\begin{equation}
\label{eq:sufficient-statistic}
\bm{b}=\{x^+_{ij}, i,j\in \cS\}.
\end{equation}
The sufficient statistic for the model \eqref{eq:model0}
is given by the union of (\ref{eq:sufficient-statistic}) and the set of
initial frequencies,  
\[
 \Bb^* = \{x^1_i , i \in \cS\} \  \cup \ \{x^+_{ij}, i,j\in \cS\}
 \supset \Bb.
\]
For $\cS=\{1,2\}$ we write the elements of $\Bb$ in \eqref{eq:sufficient-statistic}
as $b_{11},b_{12},b_{21},b_{22}$.

If we order paths appropriately and write $\Bx$ as a column vector
according to the order, $\Bb$ in (\ref{eq:sufficient-statistic}) is
written in a matrix form  
\[
\Bb = A \Bx,  
\]
where $A$ is an $|\cS|^2 \times |\cS|^T$ matrix consisting of non-negative
integers with the rows indexed by $|\cS|^2$ transitions and the columns
indexed by $|\cS|^T$ paths. The $((i,j),\omega)$ element $A$ is the number of
occurrences of the transition from $i$ to $j$ in the path $\omega$.

$A$ is the configuration defining toric ideal $I_A$ arising from the
model (\ref{eq:model}). 
A toric ideal $I_A$ is the kernel of the homomorphism of polynomial
rings 
$\psi: k[\{p(\omega),\omega\in \cS^T\}] \rightarrow 
k[\{\beta_{ij}, i,j\in \cS\}]$ defined by
\[
\psi: p(\omega) \mapsto \beta_{s_1 s_2} 
\dots \beta_{s_{T-1} s_T},  
\]
where $\{p(\omega), \omega\in \cS^T\}$ is regarded as a set of
indeterminates. 
Algebraically a Markov basis is defined as a set of generators of $I_A$
(\citet{diaconis-sturmfels}). 

The set of all contingency tables sharing $\Bb$ is called a {\it fiber}
and denoted by  
$\cF_{\Bb}=\{\Bx \in \Z^{|\cS|^T}_{\ge 0} \mid A\Bx = \Bb\}$, 
where $\Z_{\ge 0}=\{0,1,\dots\}$.
A move $\Bz$ for the model (\ref{eq:model}) is an integer array
satisfying $A \Bz =0$. 
In contingency table analysis a Markov basis is defined as a finite set 
of moves ${\cal Z}$ satisfying that for all $\Bb$ and all pairs 
$\Bx$ and $\By$ in $\cF_{\Bb}$ there exists a sequence $\Bz_1,\ldots,\Bz_K \in {\cal Z}$ such that 
\begin{equation}
 \label{eq:MB}
  \By = \Bx + \sum_{k=1}^K \Bz_k, \quad 
  \Bx + \sum_{k=1}^l \Bz_k \ge 0, \; l=1,\ldots,K.
\end{equation}
In this article we provide a Markov basis as a set of moves satisfying
(\ref{eq:MB}) for the model \eqref{eq:model} with $\cS=\{1,2\}$.

A move $\Bz$ is expressed by a difference of two contingency tables
$\Bx$ and $\By$ in the same fiber: 
\[
\Bz = \Bx - \By, \qquad  z(\omega)=x(\omega)-y(\omega), \ \omega\in \cS^T.
\]
We write 
$z^t_{ij}=x^t_{ij}-y^t_{ij}$. 

\subsection{Some classes of moves}

In this section we introduce some classes of moves for the model
(\ref{eq:model0}) discussed in \citet{th-homogeneous1}. 
Since $\Bb \subset \Bb^*$, the moves introduced here are also moves for
the model (\ref{eq:model}).  

Denote $\bm{s}_{t:t^\prime} = s_t \ldots s_{t^\prime}$
and $\bm{s}^\prime_{t:t^\prime} = s^\prime_t \ldots  
s^\prime_{t^\prime}$
for $t \le  t^\prime$. 
If $t > t^\prime$, let $\bm{s}_{t:t^\prime}$ be an empty sequence.
Let $\bar \omega=(s_1,\ldots,s_T)$ be a path satisfying 
$s_{t_0}=s_{t_1}=s_{t_2}=i$ for $1 \le t_0 < t_1 < t_2 \le T$ and 
$s_t=j \neq i$ for some $t_0 < t < t_2$.
Then consider the following swapping 
\[
 \bar\omega = 
 (\bm{s}_{1:t_0-1},\bm{s}_{t_0:t_2},\bm{s}_{t_2+1:T})
 \ \ \leftrightarrow \ \ 
 \bar \omega^\prime := 
 (\bm{s}_{1:t_0-1},\bm{s}_{t_1:t_2-1},\bm{s}_{t_0:t_1},\bm{s}_{t_2+1:T}).
\]
Then an integer array $\bm{z}=\{z(w),w \in \cS^T\}$ 
\begin{equation}
 \label{eq:deg1}
 z(\omega) = \left\{
 \begin{array}{rl}
  1 & \text{ if } \omega = \bar \omega\\
  -1& \text{ if } \omega = \bar \omega^\prime\\
  0 & \text{ otherwise}
 \end{array}
\right.
\end{equation}
forms a move for (\ref{eq:model0}).
We call this move a degree one move.

We depict a degree one move with $T=4$, $t_0=1$, $t_1=2$ and $t_2=4$ as
\begin{equation}
\label{movegraph}
\begin{tikzpicture}[baseline=-1.2cm,scale=0.8]
 \foreach \i in {1,2} {   \foreach \j in {1,...,4}    {
 \path (\j,-\i) coordinate (V\i\j);  \fill (V\i\j) circle (1pt); }}
 \draw (V11)--(V12)--(V23)--(V14);
 \draw [dotted] (V11)--(V22)--(V13)--(V14);
 \draw (0,-1) node {1};
 \draw (0,-2) node {2};
 \draw (0,-0.4) node {$s\backslash t$};
 \foreach \j in {1,...,4} { \draw (V1\j)  node [above]  {\j};}
\end{tikzpicture}\;, 
\end{equation}
where a solid line from $(i,t)$ to $(j,t+1)$ represents  
$z^t_{ij}=1$ and a dotted line from $(i,t)$ to $(j,t+1)$ represents  
$z^t_{ij}=-1$. 
We call a graph like (\ref{movegraph}) a move graph.
A node of a move graph is a pair $(i,t)$ of state $i$ and time $t$ and
an edge from $(i,t)$ to $ (j,t+1)$ represents the value of $z^t_{ij}$.  
If $\vert z^t_{ij} \vert =0$, there is no corresponding edge in the
graph. 
If $\vert z^t_{ij} \vert \ge 2$, we write the value of 
$\vert z^t_{ij} \vert$ beside the edge. 

We say that two paths $\omega=(s_1,\dots,s_T)$, $\omega'=(s'_1,\dots, s'_T)$ 
meet (or cross) at the node $(i,t)$ if
$i=s_t=s_t'$. 
If $\omega$ and $\omega'$ cross at the node $(i,t)$, consider the
swapping of these two paths like  
\begin{align*}
 &\{\bar \omega,\bar \omega'\}=
 \{\; (\bm{s}_{1:t-1},i,\bm{s}_{t+1:T}), \ 
 (\bm{s}', i,\bm{s}^\prime_{t+1:T})\; \} \nonumber\\
& \qquad  \leftrightarrow
 \{\; (\bm{s}_{1:t-1},i,\bm{s}^\prime_{t+1:T}), \ 
 (\bm{s}^\prime, i,\bm{s}_{t+1:T})\; \} 
 =:\{\tilde\omega, \tilde\omega'\}  .
\end{align*}
Then the integer array $\bm{z}$
\begin{equation}
 \label{eq:crossing}
  z(\omega) = 
  \left\{
   \begin{array}{rl}
    1 & \text{ if } \omega = \bar \omega\; \text{or}\; \bar \omega'\\
    -1& \text{ if } \omega = \tilde \omega \; \text{or}\; \tilde \omega'\\
    0 & \text{ otherwise}
   \end{array}
  \right.
\end{equation}
forms a move for (\ref{eq:model0}). 
We call this move a crossing path swapping.
As shown in \citet{th-homogeneous1}, 
a crossing path swapping is expressed by the difference of two tables
with the same edge-sign pattern and hence the move graph for a crossing
path swapping has no edge. 
As shown in \citet{th-homogeneous1}, if $\Bx$ and $\By$ are in the same
fiber such that the move graph of $\Bz$ has no edge, $\Bx$ and $\By$ are
connected by crossing path swappings. 

Suppose that $t_0 \neq t_1$. Choose (not necessarily distinct) four
paths  
$\omega_1, \omega_2,\omega_3, \omega_4$ as
\begin{align*}
\omega_1 &= (\bm{s}_{1,1:t_0-1},1,1,\bm{s}_{1,t_0+2:T}),
\ 
\omega_2 = (\bm{s}_{2,1:t_0-1},2,2,\bm{s}_{2,t_0+2:T}), 
\\
\omega_3 &= (\bm{s}_{3,1:t_1-1},1,2,\bm{s}_{3,t_1+2:T}),
\ 
 \omega_4 = (\bm{s}_{4,1:t_1-1},2,1,\bm{s}_{4,t_1+2:T}), 
\end{align*}
where 
$\bm{s}_{k,t:t'}= s_{k,t},\ldots,s_{k,t'}$.
Then we consider the swapping 
$\{\omega_1, \omega_2, \omega_3, \omega_4\} \leftrightarrow 
\{\tilde\omega_1, \tilde\omega_2, \tilde\omega_3, \tilde\omega_4\}$,
where 
\begin{align*}
 \tilde\omega_1 &= (\bm{s}_{1,1:t_0-1},1,2,\bm{s}_{2,t_0+2:T}),
 \ 
 \tilde\omega_2 = (\bm{s}_{2,1:t_0-1},2,1,\bm{s}_{1,t_0+2:T}),
 \\
 \tilde\omega_3 &= (\bm{s}_{3,1:t_1-1},1,1,\bm{s}_{4,t_1+2:T}), \ 
 \tilde\omega_4 = (\bm{s}_{4,1:t_1-1},2,2,\bm{s}_{3,t_1+2:T}).
\end{align*}
Then the integer array $\Bz$
\[
  z(\omega) = \left\{
 \begin{array}{rl}
  1 & \text{if } \omega = \omega_1,\ldots,\omega_4,\\
  -1& \text{if } \omega = \tilde \omega_1,\ldots,\tilde \omega_4,\\
  0 & \text{otherwise.}
 \end{array}
\right. 
\]
is a move for (\ref{eq:model0}) and is called 2 by 2 swap.
The corresponding move graph is depicted as
\begin{equation}
 \label{eqfig:2by2}
  \begin{tikzpicture}[baseline=-1.2cm,scale=0.8]
   \foreach \i in {1,2} {   \foreach \j in {1,...,4}    {
   \path (\j,-\i) coordinate (V\i\j);  \fill (V\i\j) circle (1pt); }}
   \path (V11) node [above] {$t_0$};
   \path (V13) node [above] {$t_1$};
   \drbl{11}{12}; \drbl{21}{22}; \drbl{13}{24}; \drbl{23}{14};
   \drwl{11}{22}; \drwl{21}{12}; \drwl{13}{14}; \drwl{23}{24};
  \end{tikzpicture}\;.
\end{equation}
For $\cS=\{1,2\}$ it can be shown that we can always choose either
$\omega_1=\omega_3$,  
$\omega_2=\omega_4$ or $\omega_1 = \omega_4$, $\omega_2=\omega_3$
(\citet{th-homogeneous1}).  
Therefore for $\cS=\{1,2\}$ a 2 by 2 swap always corresponds to a degree
two move $\Bz$.

Next we consider the following swapping for $T\ge 4$:  
\begin{align*}
 &
 (\omega_1,\omega_2) :=
 \left\{(\bm{s}_{t_0-1},1,1,2,\bm{s}_{t_0+3:T}),
 (\bm{s}^\prime_{1:t_1-1},1,2,2,\bm{s}^\prime_{t_1+3:T})\right\}  
 \nonumber \\
 & \qquad \leftrightarrow
  (\tilde \omega_1, \tilde \omega_2) :=
 \left\{ (\bm{s}_{1:t_0-1},1,2,2,\bm{s}_{t_0+3:T}),   
 (\bm{s}^\prime_{1:t_1-1},1,1,2,\bm{s}^\prime_{t_1+3:T}) \right\}.
\end{align*}
Then the integer array $\Bz$
\begin{equation}
\label{eqfig:11swap}
  z(\omega) = \left\{
 \begin{array}{rl}
  1 & \text{ if } \omega = \omega_1,\omega_2,\\
  -1& \text{ if } \omega = \tilde \omega_1,\tilde \omega_2,\\
  0 & \text{ otherwise.}
 \end{array}
\right. 
\end{equation}
also forms a move for (\ref{eq:model0}), which can be depicted as 
\[
\begin{tikzpicture}[baseline=-1.2cm,scale=0.8]
\path (0,-1.5) coordinate (midy);
\path (V11) node [above] {$t_0$};
\path (V13) node [above] {$t_1$};
\foreach \i in {1,2} {   \foreach \j in {1,...,5}    {
  \path (\j,-\i) coordinate (V\i\j);  \fill (V\i\j) circle (1pt); }}
\drbl{11}{12}; \drbl{12}{23}; \drbl{13}{24}; \drbl{24}{25};
\drwl{11}{22}; \drwl{22}{23}; \drwl{13}{14}; \drwl{14}{25};
\end{tikzpicture}
\ \ \text{\ or}\ \ 
 \begin{tikzpicture}[baseline=-1.2cm,scale=0.8]
\path (0,-1.5) coordinate (midy);
\path (V11) node [above] {$t_0$};
\path (V12) node [above] {$t_1$};
  \foreach \i in {1,2} {   \foreach \j in {1,...,4}    {
  \path (\j,-\i) coordinate (V\i\j); \fill (V\i\j) circle (1pt); }}
  \drbl{11}{12}; \drbl{12}{23}; \drbl{23}{24};
  \drwl{11}{22}; \drwl{22}{23}; \drwl{12}{13}; \drwl{13}{24};
  \draw (2.8,-1.5) node {2};
 \end{tikzpicture} \  \ .
\]

\section{Main result}

In \citet{th-homogeneous1} 
we provided a Markov basis for the model (\ref{eq:model0}) as follows.

\begin{proposition}
\label{prop:2state0}
 A Markov basis for $\cS=\{1,2\}$, $T\ge 4$, for \eqref{eq:model0}
 consists of   
 (i) degree one moves (\ref{eq:deg1}), (ii) crossing path swappings,
 (iii) 2 by 2 swaps (\ref{eqfig:2by2}),  
 (iv) moves in (\ref{eqfig:11swap}).
 For $T=3$ a Markov basis consists of the first three types of moves.
\end{proposition}

Since 2 by 2 swaps for $\cS=\{1,2\}$ corresponds to degree two moves, 
\eqref{eq:model0} has a Markov basis consisting of moves of degree at most two.

We call degree one moves in (\ref{eq:deg1}) {\em type I degree one moves}.
We note that the moves introduced in the previous section are moves 
for (\ref{eq:model0}) and hence do not alter the initial frequencies. 
For connecting fibers for the model \eqref{eq:model} we need moves
altering initial frequencies.  First we present a degree one move
altering the initial frequency.
Consider a non-flat cycle $\omega\in W_{i*j*i}$ in
\eqref{eq:iji}, vising $j$ at time $t$. Consider the following degree one move:
\begin{equation}
 \label{eq:type2-deg1}
  W_{i*j*i}\ni \omega=(s_1,\dots,s_T) \ \leftrightarrow \ 
  \omega'=(s_t, \dots, s_{T-1}, s_1, \dots, s_t) \in W_{j*i*j}.
\end{equation}
$\omega$ and $\omega'$ has the same number of transitions, but
the initial states are different.  We call this move
a {\em type II degree one move}.
An example of a type II degree one move for $T=4$ is depicted as
\begin{equation}
\label{eqfig:typeIIdeg1}
 \begin{tikzpicture}[baseline=-1.2cm,scale=0.8]
  \foreach \i in {1,2} {   \foreach \j in {1,...,4}    {
  \path (\j,-\i) coordinate (V\i\j);  \fill (V\i\j) circle (1pt); }}
  \draw (V11)--(V22)--(V13)--(V14);
  \draw [dotted] (V21)--(V12)--(V13)--(V24);
 \end{tikzpicture}\ .
\end{equation}

Next we consider a degree three move.  Let $1\le a \le b \le T-1$, $a+b \le  T-1$.
Choose $1\le u \le T-1$ such that
\[
a+u \le T-1, \quad b+(T-1-u) \le T-1.
\]
The range of such $u$ is
\[
b\le u \le T-1-a .
\]
Consider the following set of three paths: 
\[
W_1 = \{ 11\dots 1,\  \underbrace{1\dots 1}_{a}\underbrace{2\dots 2}_{T-a}, \ 
\underbrace{1\dots 1}_{b}\underbrace{2\dots 2}_{T-b}\},
\]
which consists of a  flat path  at $i=1$, 
a single-step path going from $1$ to $2$ at time $a$,  and a 
single-step path going from $1$ to $2$ at time $b$.
Choose $u \in \{b, \dots, T-1-a\}$ and let $a'=a+u$, 
$b'=b+T-1-u$.  Consider
\[
W_2 = \{ 22\dots 2,\  \underbrace{1\dots 1}_{a'}\underbrace{2\dots 2}_{T-a'}, \ 
\underbrace{1\dots 1}_{b'}\underbrace{2\dots 2}_{T-b'}\}.
\]
Then integer array $\Bz$
\begin{equation}
 \label{eq:deg3move}
  z(\omega) = 
  \left\{
   \begin{array}{rl}
    1 & \text{ if } \omega \in W_1,\\
    -1 & \text{ if } \omega \in W_2,\\
    0 & \text{ otherwise }
   \end{array}
  \right.
\end{equation}
is a move for \eqref{eq:model0}.  
We call this move a degree 3 sliding move.
An example of a degree 3 sliding move with $T=4$, $a=1$, $b=2$ and $u=3$
is depicted in Figure \ref{fig:slidingmoves} (i). 
As a degree 3 sliding move, we also consider the time reversal 
of \eqref{eq:deg3move} which involves single-step paths going from 2 to
1 as depicted in Figure \ref{fig:slidingmoves} (ii). 

\begin{figure}[htbp]
 \centering
 \begin{tabular}{cc}
  \begin{tikzpicture}[baseline=-1.2cm,scale=0.8]
   \path (0,-1.5) coordinate (midy); 
   \foreach \i in {1,...,2} { 
   \foreach \j in {1,...,4} { 
   \path (\j,-\i) coordinate (V\i\j); 
   \fill (V\i\j) circle (1pt); 
   }}
   \draw (V11)--(V22); 
   \draw[dotted] (V21)--(V22); 
   \draw[dotted] (V12)--(V13); 
   \draw (V12)--(V23); 
   \draw (V13)--(V14); 
   \draw[dotted] (V13)--(V24); 
   \draw (3.55,-1.4) node{\scriptsize 2}; 
   \draw (V23)--(V24); 
  \end{tikzpicture} & 
  \begin{tikzpicture}[baseline=-1.2cm,scale=0.8]
   \path (0,-1.5) coordinate (midy); 
   \foreach \i in {1,...,2} { 
   \foreach \j in {1,...,4} { 
   \path (\j,-\i) coordinate (V\i\j); 
   \fill (V\i\j) circle (1pt); 
   }}
   \draw (V11)--(V12); 
   \draw[dotted] (V21)--(V12); 
   \draw (1.45,-1.4) node{\scriptsize 2}; 
   \draw (V21)--(V22); 
   \draw[dotted] (V12)--(V13); 
   \draw (V22)--(V13); 
   \draw (V23)--(V14); 
   \draw[dotted] (V23)--(V24); 
  \end{tikzpicture}   \vspace{0.2cm}\\
  (i) & (ii)
 \end{tabular}
 \caption{Sliding moves}
 \label{fig:slidingmoves}
\end{figure}
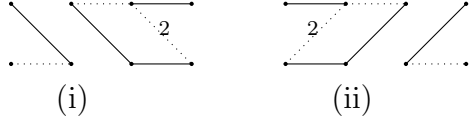
Now we state the main theorem of this paper.

\begin{theorem}
\label{thm:main}
 The set of type II degree one moves of \eqref{eq:type2-deg1} and the set
 of degree 3 sliding moves in \eqref{eq:deg3move}, in addition to moves
 in Proposition \ref{prop:2state0}, constitutes a Markov basis for
 \eqref{eq:model0}.
\end{theorem}

We give a proof of this theorem for $T\ge 4$ in the next section.  
Note that the above statement 
also covers the case $T=3$, namely, for $T=3$, we do not need moves of type (iv) in
Proposition \ref{prop:2state0}.

At this point we briefly discuss the case of $T=3$.  Although we can prove
Theorem \ref{thm:main} for the case $T=3$ 
along the lines of the next section, it is trivial to use 4ti2 (\citet{4ti2}) and
verify Theorem \ref{thm:main}.  It is interesting to note that 
the degree 3 sliding move for $T=3$ 
\[
\{111,122,122\}\  \leftrightarrow \ \{ 112,112,222\}
\]
is indispensable (\citet{takemura-aoki-2004aism}, 
\citet{ohsugi-hibi-2005jaa}).
Hence every Markov basis has to contain a degree three move for $T=3$.

\section{Proof of the main theorem}
In this section we give a proof of Theorem \ref{thm:main} for $T\ge 4$.
We first take care of two types of special fibers.  
Then we employ a distance reduction argument as in
\citet{th-homogeneous1} and \citet{takemura-aoki-2005bernoulli}.

Consider a fiber with $b_{11}=0$ or $b_{22}=0$.  
We have the following lemma.
\begin{lemma}
\label{lem:b11=0}
 Fibers with $b_{11}=0$ or $b_{22}=0$ are connected by type I degree
 one moves and crossing path swappings. 
\end{lemma}

\begin{proof}
 By symmetry consider the case $b_{11}=0$.
 Consider an arbitrary fiber $\cF_\Bb$ with $b_{11}=0$ and arbitrary
 $\Bx\in \cF_\Bb$. 
 It suffices to show that we can transform $\bm{x}$ to
 a unique $\Bx^* \in \cF_\Bb$ by type I degree one moves and crossing
 path swappings. 
 This can be accomplished as follows.  
 We can easily check that by applying type I degree one moves and
 crossing path swappings,  we can transform paths 
 $w \in W_{1* \cdots *2}$ in $\bm{x}$ to following three types of paths,    
 \begin{enumerate}
  \item $(1212 \cdots 12)$ when $T$ is even or $(1212 \cdots 122)$ when 
	$T$ is odd; 
	\[
	\begin{tikzpicture}[baseline=-1.1cm,scale=0.7]
	 \path (0,-1.5) coordinate (midy); 
	 \foreach \i in {1,...,2} { 
	 \foreach \j in {1,...,6} { 
	 \path (\j,-\i) coordinate (V\i\j); 
	 \fill (V\i\j) circle (1pt); 
	 }}
	 \draw (V11)--(V22)--(V13)--(V24)--(V15)--(V26); 
	\end{tikzpicture}
	\quad \text{ or }
	\begin{tikzpicture}[baseline=-1.1cm,scale=0.7]
	 \path (0,-1.5) coordinate (midy); 
	 \foreach \i in {1,...,2} { 
	 \foreach \j in {1,...,7} { 
	 \path (\j,-\i) coordinate (V\i\j); 
	 \fill (V\i\j) circle (1pt); 
	 }}
	 \draw (V11)--(V22)--(V13)--(V24)--(V15)--(V26)--(V27); 
	\end{tikzpicture}
	\]
  \item $(1212 \cdots 12 22 \cdots 2)$; 
	\[
	\begin{tikzpicture}[baseline=-1.1cm,scale=0.7]
	 \path (0,-1.5) coordinate (midy); 
	 \foreach \i in {1,...,2} { 
	 \foreach \j in {1,...,6} { 
	 \path (\j,-\i) coordinate (V\i\j); 
	 \fill (V\i\j) circle (1pt); 
	 }}
	 \draw (V11)--(V22)--(V13)--(V24)--(V25)--(V26); 
	\end{tikzpicture}\quad
	\]
  \item $(12 22 \cdots 2)$; 
	\[
	\begin{tikzpicture}[baseline=-1.1cm,scale=0.7]
	 \path (0,-1.5)	 coordinate (midy); 
	 \foreach \i in {1,...,2} { 
	 \foreach \j in {1,...,6} { 
	 \path (\j,-\i) coordinate (V\i\j); 
	 \fill (V\i\j) circle (1pt); 
	 }}
	 \draw (V11)--(V22)--(V23)--(V24)--(V25)--(V26); 
	\end{tikzpicture}\quad
	\]
 \end{enumerate}
 so that the number of type 2 paths is at most one. 
 In the same way we can transform paths $w \in W_{2* \cdots *2}$ in
 $\bm{x}$ to 
 \begin{enumerate}
  \item $(2121 \cdots 2122)$ when $T$ is even and 
	$(2121 \cdots 212)$ when $T$ is odd; 
  \item $(2121 \cdots 2122 \cdots 2)$; 
 \item $(22 \cdots 2)$; 
 \end{enumerate}
 so that the number of type 2 paths is at most one. 
 The paths $w \in W_{1* \cdots *1}$ in $\bm{x}$ are transformed to 
 \begin{enumerate}
  \item $(1212 \cdots 12)$ when $T$ is even and 
	$(1212 \cdots 1221)$ when $T$ is odd; 
  \item $(1212 \cdots 1222 \cdots 21)$; 
 \item $(122 \cdots 21)$; 
 \end{enumerate}
 so that the number of type 2 paths is at most one. 
 The paths $w \in W_{2* \cdots *1}$ in $\bm{x}$ are transformed to 
 \begin{enumerate}
  \item $(2121 \cdots 21)$ when $T$ is even and 
	$(2121 \cdots 21221)$ when $T$ is odd; 
  \item $(2121 \cdots 2122 \cdots 21)$; 
 \item $(22 \cdots 21)$; 
 \end{enumerate}
 so that the number of type 2 paths is at most one. 
 Then the resulting contingency table is uniquely defined.
\end{proof}

Next consider a fiber with $b_{12}=0$ or $b_{21}=0$.  
\begin{lemma}
 \label{lem:b12=0}
 Fibers with $b_{12}=0$ or $b_{21}=0$ are connected by 
 moves of type (iv) in Proposition \ref{prop:2state0}, crossing path swappings 
 and degree 3 sliding moves.
\end{lemma}

\begin{proof}
 By symmetry consider the case $b_{21}=0$.  In this case every path
 of a contingency table is either a flat path or a single-step path going
 from 1 to 2. In particular the number of single-step paths is $b_{12}$.
 Note that a contingency table of a 
 fiber is determined by choosing positions of 
 transitions for the single-step paths from 1 to 2.   A single-step path
 going from 1 to 2 at time $t$ contains $t-1$ transitions $1\rightarrow 1$
 and $T-t-1$ transitions $2\rightarrow 2$. Therefore once we choose
 positions of the transitions, the number of transitions $1\rightarrow 1$
 and $2\rightarrow 2$ contained in these single-step paths are determined.  The
 remaining  number of transitions $1\rightarrow 1$ and $2\rightarrow 2$ 
 belong to flat paths.  Therefore the remaining numbers have to be multiples of $T-1$.
 This implies that fiber is partitioned into subsets
 by the number of flat paths at $1$ or at $2$.
 Note that two contingency tables  $\Bx, \By$ with the same number of
 flat paths at $1$ and at $2$ have the same initial frequencies. 
 Then these $\Bx$, $\By$ are connected by moves of type (iv) in Proposition
 \ref{prop:2state0} and crossing path swappings.   

 Suppose that initial frequencies of $\Bx$ and $\By$ are different.
 We note that $b_{11} + b_{22} + b_{12} = n(T-1)$. 
 Let $\alpha\in \{0,1,\dots,T-2\}$ be the remainder of $b_{11}$ when it
 is divided by $T-1$ and $\beta$ be the remainder of $b_{22}$ when it is
 divided by $T-1$.
 Then $\alpha + \beta + b_{12} = n'(T-1)$ for some integer $n' \le n$. 

 Since $0 \le \alpha,\beta \le T-2$, 
 $\alpha + \beta + b_{12} = (T-1)$ for $b_{12}=1$.
 Then we can see that a fiber contains a single contingency table when 
 $b_{12}=1$.

 Now consider the case that $b_{12}=2$.  
 It can be easily verified that if 
 $\alpha=T-2$, $\beta=T-2$ and the two transitions
 from $1$ to $2$ have to occur at $t$ and $T-t$ ($1\le t \le T-1$).
 Then the number of transitions $1\rightarrow 1$ and 
 $2 \rightarrow 2$ in the single-step paths are both $T-2$.
 Since $\Bx$ and $\By$ are in the same fiber, the number of flat paths
 are the same.
 Hence the initial frequency of $\Bx$ and $\By$ is uniquely determined.  
 Therefore $\Bx$ and $\By$ in the same fiber are connected by moves of
 type (iv) in Proposition  \ref{prop:2state0}.  
 If  $\alpha < T-2$ or $\beta < T-2$, then 
 the number of flat paths at 1 in the fiber takes two consecutive
 integer values. 
 Without loss of generality consider the case $x(11\dots1)=y(11\dots1)+1$.
 Then we can apply  a degree 3 sliding move $\Bz$ to $\Bx$ such that
 $\Bx'(11\dots1)=\Bx(11\dots1)-1$ where $\Bx'=\Bx +\Bz$. Then
 $\Bx'$ and $\By$ are connected by moves of type (iv) in Proposition
 \ref{prop:2state0}. 

  For the case $b_{12} \ge 3$, there are more than two single-step paths going
 from 1 to 2. In this case we pick two such paths and apply moves of type (iv) of
 Proposition \ref{prop:2state0}, so that the locations of  transitions from 1 to 2 of 
 these two paths are far apart from those of other paths. 
 Then we can slide these two paths to alter the initial frequency.
 This proves the lemma.
\end{proof}

After proving above two lemmas, it suffices to consider fibers 
satisfying 
\begin{equation}
 \label{ineq:positive.suff.stat}
  b_{11} > 0, \; 
  b_{12} > 0, \;   
  b_{21} > 0, \; 
  b_{22} > 0.
\end{equation}
We now employ distance reduction argument.  
Let $\Bx$ and $\By$ be two contingency tables of the same fiber 
and $\Bz := \Bx - \By$. 
Define $|\Bz| := \sum_{t}\sum_{i,j} |z_{ij}^t|$. 
If $\Bx$ and $\By$ have the same initial frequencies, then they are
connected by moves in Proposition \ref{prop:2state0}.  
We note that 
\[
b_{21} - b_{12} = x_1^T - x_1^1 = y_1^T - y_1^1, 
\]
which implies $z_1^1 = z_1^T$. 
In the same way we have $z_2^1 = z_1^T$. 
Therefore, without loss of generality we assume 
\[
z^1_1 =a, \; z^T_1=a, \; z_2^1=-a, \; z^T_2=-a,   \quad a>0.
\]

\begin{lemma}
\label{lem:3}
 Assume (\ref{ineq:positive.suff.stat}) and $T>3$. 
 Suppose that $a=z^1_1 > 0$ can not be decreased by type II degree one moves.  
 Then $x(11\dots 1)>0$ and $y(22\dots 2)>0$ or 
 by these moves we can transform $\Bx$ and $\By$ to $\Bx'$ and
 $\By'$ such that $x'(11\dots 1)>0$ and $y'(22\dots 2)>0$.
\end{lemma}

\begin{proof}
 Suppose that there exists a path 
 $w = (s_1,\ldots,s_{T}) \in W_{1*2*1}$ in $\bm{x}$ such that $s_t = 2$
 for $1<t<T$. 
 Then by a type II degree one move (\ref{eq:type2-deg1}) with $i=1$ and
 $j=2$, we can reduce $a$.
 Therefore we can assume that all paths in $W_{1*\cdots*1}$ are
 $(11\cdots1)$. 
 By symmetry, we can also assume that all paths of $\By$ in 
 $W_{2*\cdots*2}$ are $(22\cdots2)$.

 Suppose that $x(11 \cdots 1)=0$. 
 By the above argument, $\Bx$ has no path in $W_{1*\cdots*1}$. 
 Hence for any path in $\Bx$, $s_1=1$ implies $s_T=2$ and
 $s_T=1$ implies $s_1=2$. 
 Then we note that $z_2^1 = z_2^T=-a$ implies $y(22\cdots 2)>0$. 
 Therefore either $x(11 \cdots 1)>0$ or $y(22\cdots 2)>0$ is satisfied.  
 We now assume $x(11 \cdots 1)>0$ without loss of generality.
 
 Suppose that $y(22\cdots2)=0$.
 Then for any path in $\By$, $s_1=2$ implies $s_T=1$ and $s_T=2$ implies
 $s_1=1$. 
 Suppose that $\bar \omega = (s_1,\ldots,s_{T}) \in W_{2*\cdots*1}$ 
 is a path in $\By$.   
 Since $z_2^1=z_2^T=-a < 0$, there has to exist a path 
 $\bar \omega^\prime = (s^\prime_1,\ldots,s^\prime_{T}) \in
 W_{1*\cdots*2}$ in $\By$. 

 If $\bar \omega$ and $\bar \omega'$ meet at $t$, 
 by applying a crossing path swapping $\bm{z}$ in
 (\ref{eq:crossing}), 
 $\bm{y}$ is transformed to 
 $\bm{y}^\prime = \{y'(w), w \in \{1,2\}^T\} \in \cF_\Bb$ 
 such that $y^\prime(\tilde \omega) >0$ and 
 $\tilde \omega = (s_1,\ldots,s_t,s'_{t+1},\ldots,s'_T) 
 \in W_{2* \cdots *2}$. 
 Then by the above argument we can assume that 
 $y^\prime(22\cdots 2) >0$. 

 Next we consider the case where $\bar \omega$ does not meet 
 any $\bar \omega^\prime \in W_{1*\cdots*2}$ in $\By$. 
 Then $\bm{y}$ has only one path in each of $W_{2* \cdots *1}$ and  
 $W_{1* \cdots *2}$. 
 Let $\bar \omega$ and $\bar \omega^\prime$ be such paths.  
 Suppose that $\bar \omega=(2121 \cdots 21)$ and 
 $\bar \omega^\prime = (1212 \cdots 12)$.
 Then from the assumptions that $b_{11} > 0$ and 
 $b_{22} > 0$  we have $y^\prime(22\cdots 2) >0$.

 Suppose that both $\bar \omega$ and $\bar \omega^\prime$ are
 single-step paths.  
 If $b_{22}=T-2$, 
 \[
 y(\omega) = 
 \left\{
 \begin{array}{ll}
  1, & \text{ if } \omega = \bar \omega \text{ or } \bar \omega^\prime\\
  0, & \text{ otherwise}
 \end{array}
 \right.
 \]
 and $b_{12}=b_{21}=1$. 
 Then $\Bx$ has to contain two single-step paths $\tilde \omega$ and 
 $\tilde \omega^\prime$ which does not meet each other. 
 $(\bar \omega, \bar \omega^\prime)$ is transformed to 
 $(\tilde \omega, \tilde \omega^\prime)$ by a 2 by 2 swap (\ref{eqfig:2by2}).
 Hence $|\bm{z}|$ is reduced. 
 
 If $b_{22} > T-2$, there has to exist another path 
 $\bar \omega''=(s''_1,\ldots,s''_T)$ in $\By$ such that 
 $s''_{t} = s''_{t+1} =2$ for some $t$.
 By a 2 by 2 swap, $\bar \omega$ and $\bar \omega^\prime$ are
 transformed to single-step paths at $t$.  
 Denote them by 
 $\hat \omega=(22\cdots 211\cdots 1)$ and 
 $\hat \omega^\prime = (11\cdots 122\cdots 2)$.
 Then apply crossing path swappings as follows, 
 \begin{align*}
  & \{\hat \omega, \bar \omega''\} = 
  \{
  (22 \cdots2 11 \cdots 1), 
  (\bm{s}''_{1:t}, \bm{s}''_{t+1:T})
  \}\\
  & \qquad \leftrightarrow 
  \{
  (22 \cdots 2,\bm{s}''_{t+1:T}), 
  (\bm{s}''_{1:t},11 \cdots 1),
  \} 
 \end{align*}
 \begin{align*}
  & \{\hat \omega', (22 \cdots 2,\bm{s}''_{t+1:T}) \} = 
  \{
  (22 \cdots 22,\bm{s}''_{t+2:T}), 
  (11\cdots 122\cdots 2)
  \}\\
  & \qquad \leftrightarrow \{
  (22\cdots 2), (11\cdots 12 \bm{s}''_{t+2:T}),
  \} 
 \end{align*}
 Therefore $\By$ is transformed to $\By'$ such that 
 $y^\prime(22\cdots 2) >0$. 

 Suppose that both $\bar\omega$ and $\bar\omega^\prime$ are not single-step
 paths. 
 Then we can easily see that $\bar \omega$ is transformed to another path 
 $\bar \omega''=(s''_1,\ldots,s''_T)$ by a type I degree one move. 
 Then $\omega'$ and $\omega''$ meet somewhere.
 Therefore in the same way as the above argument by a crossing path
 swapping to $\bm{y}$, $\By$ is transformed to $\By'$ such that 
 $y^\prime(22\cdots 2) >0$.
\end{proof}

\begin{lemma}
\label{lem:4}
 Assume (\ref{ineq:positive.suff.stat}) and $T>3$. 
 Suppose that $a=z^1_1>0$ can not be decreased by the moves of Theorem
 \ref{thm:main} except for degree 3 sliding moves. 
 Then by these moves we can transform $\Bx$ and $\By$ to $\Bx'$ and
 $\By'$ which consist of flat paths and single-step paths only.
\end{lemma}

\begin{proof}
 By Lemma \ref{lem:3}, we can assume that $x(11\cdots 1) >0$. 
 By the argument in the proof of Lemma \ref{lem:3}, we can assume
 that paths in $\bm{x}$ which start and end at $1$ are flat paths at 
 $1$ $(11 \cdots 1)$.
 In the same way we can easily show that paths in $\bm{x}$ which start
 at $1$ and end at $2$ are assumed to be single-step paths 
 $(11 \cdots 1 2 \cdots 2)$.    
 We can also assume that paths in $\bm{x}$ which ends at $1$ are flat
 paths at $1$ $(11 \cdots 1)$ or single-step paths 
 $(22 \cdots 2 1 \cdots 1)$.  
 
 Next we consider a path $\omega$ in $\bm{x}$ which starts and ends at
 $2$.   
 Suppose that $\omega$ is not a flat path at $2$. 
 Then there exists $1 < t_0 \le t_1 < T$ such that 
 $\bm{s}_{1:t_0-1}=(22 \cdots 2)$, $s_{t_0}=s_{t_1}=1$, 
 $\bm{s}_{t_1+1:T}=(22 \cdots 2)$.
 We now suppose that there exists $t_0 < t_2 < t_1$ such that 
 $s_{t_2}=2$. 
 Then by crossing path swapping of $\omega$ and a flat path $\omega'$ at
 $1$,
 we can transform $(\omega,\omega')$ to 
 $\tilde \omega \in W_{2*1*2}$ and $\tilde \omega' \in W_{1*2*1}$ as
 follows.  
 \begin{align*}
  &
 \begin{tikzpicture}[scale=0.5]
  \path (0,-1.5)	 coordinate (midy); 
  \foreach \i in {1,...,2} { 
  \foreach \j in {1,...,7} { 
  \path (\j,-\i) coordinate (V\i\j); 
  \fill (V\i\j) circle (1pt); 
  }}
  \draw (V21)--(V22)--(V13)--(V24)--(V15)--(V26)--(V27); 
 \end{tikzpicture}\quad , 
 \begin{tikzpicture}[scale=0.5]
  \path (0,-1.5)	 coordinate (midy); 
  \foreach \i in {1,...,2} { 
  \foreach \j in {1,...,7} { 
  \path (\j,-\i) coordinate (V\i\j); 
  \fill (V\i\j) circle (1pt); 
  }}
  \draw (V11)--(V12)--(V13)--(V14)--(V15)--(V16)--(V17); 
 \end{tikzpicture}\\
 & \qquad \qquad \qquad \leftrightarrow
 \begin{tikzpicture}[scale=0.5]
  \path (0,-1.5)	 coordinate (midy); 
  \foreach \i in {1,...,2} { 
  \foreach \j in {1,...,7} { 
  \path (\j,-\i) coordinate (V\i\j); 
  \fill (V\i\j) circle (1pt); 
  }}
  \draw (V21)--(V22)--(V13)--(V14)--(V15)--(V26)--(V27); 
 \end{tikzpicture}\quad ,
  \begin{tikzpicture}[scale=0.5]
  \path (0,-1.5)	 coordinate (midy); 
  \foreach \i in {1,...,2} { 
  \foreach \j in {1,...,6} { 
  \path (\j,-\i) coordinate (V\i\j); 
  \fill (V\i\j) circle (1pt); 
  }}
  \draw (V11)--(V12)--(V13)--(V24)--(V15)--(V16)--(V17); 
  \end{tikzpicture}\quad .
 \end{align*}
 Then we can reduce $a$ by applying a type II degree one move to 
 $\tilde \omega'$.  
 Hence we can assume that $\bm{s}_{t_0:t_1}=(11 \cdots 1)$. 
 
 Since $x(11 \cdots 1) > 0$, 
 we can apply a crossing path swapping of 
 $\omega$ and a flat path $(11 \cdots 1)$, we can transform them to two
 single-step paths as 
 \begin{align*}
  &
 \begin{tikzpicture}[scale=0.5]
  \path (0,-1.5)	 coordinate (midy); 
  \foreach \i in {1,...,2} { 
  \foreach \j in {1,...,6} { 
  \path (\j,-\i) coordinate (V\i\j); 
  \fill (V\i\j) circle (1pt); 
  }}
  \draw (V21)--(V22)--(V13)--(V14)--(V25)--(V26); 
 \end{tikzpicture}\quad , 
 \begin{tikzpicture}[scale=0.5]
  \path (0,-1.5)	 coordinate (midy); 
  \foreach \i in {1,...,2} { 
  \foreach \j in {1,...,6} { 
  \path (\j,-\i) coordinate (V\i\j); 
  \fill (V\i\j) circle (1pt); 
  }}
  \draw (V11)--(V12)--(V13)--(V14)--(V15)--(V16); 
 \end{tikzpicture}\\
 & \qquad \qquad \qquad \leftrightarrow
 \begin{tikzpicture}[scale=0.5]
  \path (0,-1.5)	 coordinate (midy); 
  \foreach \i in {1,...,2} { 
  \foreach \j in {1,...,6} { 
  \path (\j,-\i) coordinate (V\i\j); 
  \fill (V\i\j) circle (1pt); 
  }}
  \draw (V21)--(V22)--(V13)--(V14)--(V15)--(V16); 
 \end{tikzpicture}\quad ,
  \begin{tikzpicture}[scale=0.5]
  \path (0,-1.5)	 coordinate (midy); 
  \foreach \i in {1,...,2} { 
  \foreach \j in {1,...,6} { 
  \path (\j,-\i) coordinate (V\i\j); 
  \fill (V\i\j) circle (1pt); 
  }}
  \draw (V11)--(V12)--(V13)--(V14)--(V25)--(V26); 
  \end{tikzpicture}\quad .
 \end{align*}
 Therefore we can transform $\bm{x}$ to $\bm{x}'$.
 In the same way we can show that $\bm{y}$ is transformed to 
 $\bm{y}'$. 
\end{proof}

By Lemma \ref{lem:4} we assume that $\Bx$ and $\By$ consist of flat
paths and single-step paths only.    
Now by applying degree 3 sliding moves and adjusting the number of flat 
paths at 1 as in Lemma \ref{lem:b12=0}, 
we can adjust the initial frequencies of $\Bx$ and $\By$.
This proves Theorem \ref{thm:main}.

\section{A numerical example}

In this section we give a numerical example for testing THMC model
without initial parameters against THMC model with initial parameters
using a Markov basis derived in the previous section.    

Table \ref{tab:klotz} refers to the number of sequences of the sexes of
the first four children in order of birth for selected Amish families
which had more than or equal to four children (\citet{klotz1972}). 
In the selected families, parents were born before or up to 1910.
The families with multiple birth are eliminated.
Klotz \citet{klotz1972} found some Markov dependence between sexes of
consecutive children. 

Here we consider the test of the uniformity of the initial
parameters $\gamma_1 = \gamma_2$. 
In order to verify the uniformity of the initial distribution, 
we test the goodness-of-fit of THMC model without initial parameters
against THMC model $H_1$ via Markov basis technique.  
We use likelihood ratio statistic $L$ as a test statistics.
For data in Table 1 we have $L = 0.1219$. 

We computed the exact distribution of $L$ via MCMC with a Markov basis
derived in Theorem 3.1. 
We sampled 10,000 tables after 5,000 burn-in steps.
The histogram of the sampling distribution of $L$ is shown in 
Figure \ref{fig:MCMC}.
The solid line represents the density function of the asymptotic
$\chi^2$ distribution with degrees of freedom $1$. 
The asymptotic $p$-value and the exact $p$-value are 
$0.7270$ and $0.6461$, respectively and hence $H_0$ is accepted.\\

\begin{table}[htbp]
 \centering
 \caption{The sexes of the first four children in order of birth for
 selected Amish families}
 \label{tab:klotz}
  \begin{tabular}{cc|cc}\hline
   MMMM & 8 & FMMM & 13 \\ \hline
   MMMF & 14& FMMF & 11\\
   MMFM & 13& FMFM & 9\\
   MMFF & 19& FMFF & 9\\
   MFMM & 11& FFMM & 10\\
   MFMF & 9 & FFMF & 8\\
   MFFM & 11& FFFM & 9\\
   MFFF & 13& FFFF & 10\\ \hline
  \end{tabular}\\
 M : male, \; F : female
\end{table}

\begin{figure}
 \centering
 \includegraphics[scale=0.4]{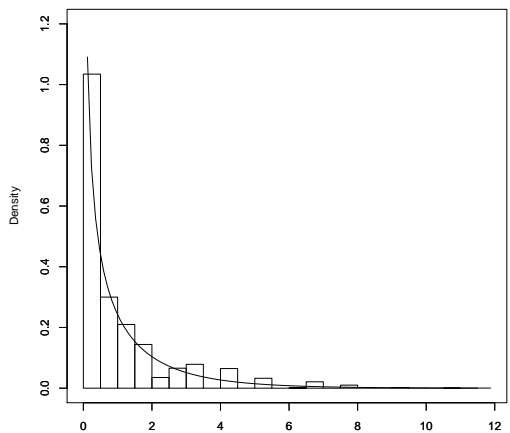}\\
 \caption{Sampling distribution of $L$ via MCMC}
 \label{fig:MCMC}
\end{figure}

\section{Concluding remarks}

We derived a Markov basis for THMC model without initial parameters 
\eqref{eq:model} for $\cS=\{1,2\}$
and arbitrary $T\ge 3$.  The basis consists of moves of degree at most three
and the types of moves  are common for all $T\ge 4$.
For the model \eqref{eq:model0} we had similar ``finiteness'' result
in \citet{th-homogeneous1} with a Markov basis consisting of moves of degree at most two.

Each fiber of \eqref{eq:model0} is a subset of a fiber in \eqref{eq:model}.
This corresponds to the fact that 
\eqref{eq:model} is a submodel of \eqref{eq:model0}, such that
the sufficient statistic for \eqref{eq:model} is a linear function of 
the sufficient statistic for \eqref{eq:model0}.
Type II degree one moves and the degree 3 sliding moves are needed to
connect fibers of \eqref{eq:model0} in each fiber of \eqref{eq:model}.
It is of interest to consider other nested toric statistical models
and identify moves which are needed to cross fibers of a larger model
within each fiber of a smaller model.

\bibliographystyle{plainnat}
\bibliography{Hara-Takemura-homogeneous}

\end{document}